\documentclass[a4paper]{amsart}

\usepackage[shortlabels]{enumitem}

\usepackage[dvipsnames]{xcolor}
\usepackage[backref]{hyperref}
\hypersetup{colorlinks=true,linkcolor=Maroon,citecolor=OliveGreen}
\usepackage{mathabx}
\usepackage{mathtools}
\usepackage{enumitem}
\usepackage{xfrac}

\usepackage{cleveref}

\usepackage{microtype}

\newtheorem{theorem}{Theorem}[section]
\newtheorem*{theorem*}{Theorem}
\newtheorem{lemma}[theorem]{Lemma}
\newtheorem{proposition}[theorem]{Proposition}
\newtheorem{corollary}[theorem]{Corollary}

\theoremstyle{definition}

\newtheorem{example}[theorem]{Example}

\newtheorem*{conjecture*}{Conjecture}

\usepackage{amsmath,amsfonts,amsthm,mathtools,commath,amssymb}

\usepackage{bbm}
\def\1{\mathbf{1}}
\def\F{\mathbf{F}}

\DeclareMathOperator{\Mat}{M}
\DeclareMathOperator{\characteristic}{char}
\DeclareMathOperator{\slfrak}{\mathfrak{sl}}
\DeclareMathOperator{\pslfrak}{\mathfrak{psl}}
\DeclareMathOperator{\diag}{diag}
\DeclareMathOperator{\tr}{Tr}

\DeclareMathOperator{\GL}{GL}
\DeclareMathOperator{\transpose}{T}
\DeclareMathOperator{\Fr}{Fr}
\DeclareMathOperator{\Div}{Div}
\DeclareMathOperator{\mult}{mult}
\DeclareMathOperator{\ad}{ad}
\DeclareMathOperator{\Gal}{Gal}

\begin{document}
\baselineskip=13pt % comfortable reading

\title{Two-generation of traceless matrices over finite fields}

\author{Omer Cantor}
\address{Omer Cantor, Department of Mathematics, University of Haifa, 199 Abba Khoushy Avenue, Haifa, Israel}
\email{ocantor@proton.me}

\author{Urban Jezernik}
\address{Urban Jezernik, Faculty of Mathematics and Physics, University of Ljubljana and Institute of Mathematics, Physics, and Mechanics, Jadranska 19, 1000 Ljubljana, Slovenia}
\email{urban.jezernik@fmf.uni-lj.si}

\author{Andoni Zozaya}
\address{Andoni Zozaya, Faculty of Mathematics and Physics, University of Ljubljana, Jadranska 19, 1000 Ljubljana, Slovenia}
\email{andoni.zozaya@fmf.uni-lj.si}

\thanks{UJ is supported by the Slovenian Research Agency program P1-0222 and grants J1-50001, J1-4351, J1-3004, N1-0217. AZ is supported by the Slovenian Research Agency programs P1-0222, P1-0294 and grants J1-50001, J1-4351, N1-0216.}

\begin{abstract}
  We prove that the Lie algebra $\slfrak_n(\F_q)$ of traceless matrices over a finite field of characteristic $p$ can be generated by $2$ elements with exceptions when $(n,p)$ is $(3,3)$ or $(4,2)$. In the latter cases, we establish curious identities that obstruct $2$-generation.
\end{abstract}

\maketitle

\section{Introduction}
\label{sec:introduction}

In the study of algebraic structures, the concept of generation -- specifically, the ability of a small number of elements to generate an entire structure -- plays a crucial role in understanding their complexities and capabilities. A particularly notable generation property observed across \emph{simple} algebraic structures is their tendency to be generated by merely \emph{two} elements.

This $2$-generation phenomenon is well-established across various classical algebraic structures over \emph{classical fields} such as the complex numbers. For simple associative algebras, this property is underscored by the classical Burnside irreducibility theorem, which posits that over algebraically closed fields, any two matrices without a common proper invariant subspace generate the whole matrix algebra. Similarly, for simple algebraic groups, it is a result of Guralnick \cite[Theorem 3.3]{guralnick1998some} that over an algebraically closed field of any characteristic, a simple algebraic group $G$ is topologically $3/2$-generated, meaning that for every nontrivial $x$ there exists $y$ so that $\langle x, y \rangle$ is dense in $G$. Furthermore, simple Lie algebras over any field of characteristic $0$ have been shown to be $2$-generated by Kuranishi \cite{kuranishi1951everywhere} (see also \cite{bahturin1976} for a more generic approach), and explicit generating pairs were found by Grozman and Leites \cite{leites1996}. This was extended to $3/2$-generation in characteristic $0$ by Ionescu \cite{ionescu1976generators} and in infinite fields of positive characteristic $p > 3$ by Bois \cite{bois2009generators}. More recently, Chistopolskaya \cite{chistopolskaya2018nilpotent} has explored nilpotent pairs of generators in traceless matrices.

Within the context of \emph{finite fields}, there are analogous results for simple associative algebras, such as explicit generating pairs for matrices over integers provided by Petrenko and Sidki \cite[Example 2.6]{petrenko2007pairs}, which project to matrices over finite fields. Additionally, Neumann and Praeger \cite{neumann1995cyclic} have demonstrated that random pairs of matrices over a finite field generate the whole algebra with high probability as the field's size increases. Similar advancements have been made for finite simple groups, with Steinberg \cite{steinberg1962generators} proving that finite simple Chevalley groups are $2$-generated, and subsequently Aschbacher and Guralnick \cite{aschbacher1984some} establishing $2$-generation for all finite simple groups. Research into $2$-generation of Lie algebras over finite fields remains comparatively underexplored. This paper aims to address that gap.

\begin{theorem*}
  Let $\F_q$ be a finite field of characteristic $p$. Then the simple Lie algebra $\pslfrak_n(\F_q)$ is $2$-generated as long as $(n,p) \neq (3,3), (4,2)$.
\end{theorem*}

We show that both $\pslfrak_3(\F)$  and $\pslfrak_4(\F)$ are, intriguingly, \emph{not} $2$-generated when $\F$ is any field of characteristic $3$ and $2$ respectively. However, these Lie algebras do allow for $3$-generation. These are, to the best of our knowledge, the first known examples of simple Lie algebras over infinite fields that are not generated by two elements, thus answering \cite[Question 2.3]{guralnick2006thompson} as well as \cite[Problem 2]{chistopolskaya2018nilpotent}.

The strategy of the proof is outlined as follows. Our argument rests on the concept of \emph{consistent matrices} (\Cref{sec:consistent}) in $\slfrak_n(\F_q)$. We explain what these matrices are, where they come from and how they can be used to prove $2$-generation of $\slfrak_n(\F_q)$, and consequently of $\mathfrak{psl}_n(\F_q)$, as long as they exist. When $n$ is fixed and the field $\F_q$ is large enough in terms of $n$ (this is the case of \emph{bounded rank}, \Cref{sec:bounded_rank}), we can produce these matrices from suitable subsets of integers related to distinct sum sets from additive combinatorics. On the other hand, when $n$ tends to infinity and the field is small (this is the case of \emph{unbounded rank}, \Cref{sec:unbounded_rank}), we produce generating pairs from companion matrices of suitable elements in extensions of $\F_q$. In the simple case, when $p$ does not divide $n$, normal elements do the job. In the modular case, when $p$ divides $n$, we replace normal elements by developing their traceless version which we call sharply traceless elements. \emph{Even characteristic} (\Cref{sec:even_characteristic}) requires special care and a further step to reach $2$-generation. The \emph{exceptional cases} (\Cref{sec:exceptions}) that our method does not cover are the Lie algebras $\pslfrak_3(\F_q)$ and $\pslfrak_4(\F_q)$ in characteristics $3$ and $2$ respectively. We demonstrate that these algebras are not $2$-generated by establishing a curious identity between any two elements in these Lie algebras that obstructs $2$-generation.

\section{Consistent matrices}
\label{sec:consistent}

An important ingredient of our proof of $2$-generation consists of \emph{consistent matrices}. These are diagonal matrices
\[
  D = \diag(\lambda_1, \lambda_2, \dots, \lambda_n) \in \slfrak_n(\F)
\]
over a field $\F$ satisfying the condition
\begin{equation}
  \label{eq:consistency_def}
  \lambda_i - \lambda_j = \lambda_k -\lambda_l \ \textup{ if and only if } \ (i,j) = (k,l) \text{ or } (i,k) = (j,l).
\end{equation}
The set of diagonal elements of a consistent matrix $D$ is called a \emph{consistent set}.
For example, over a field $\F$ of characteristic $0$, the following is a consistent set for any $n \geq 2$:
\begin{equation}
\label{eq:consistent_example}
\left\{ 1, 2, 2^2, 2^3, \dots, 2^{n-2}, 1 - 2^{n-1} \right\}.
\end{equation}
Note that when $\characteristic(\F) = 2$, one has $\lambda_i - \lambda_j = \lambda_j - \lambda_i$ for any $i,j$, so no consistent set exists. We will deal with even characteristic separately in \Cref{sec:even_characteristic}.

Consistent matrices were implicitly used by Kuranishi \cite{kuranishi1951everywhere} in the study of $2$-generation of semisimple Lie algebras of characteristic $0$, and were further inspected by Chistopolskaya \cite{chistopolskaya2018nilpotent} who was looking for pairs of nilpotent generators of Lie algebras over infinite fields. The relevance of consistent matrices to $2$-generation is as follows. Let 
\[
  \1 = E - I \in \slfrak_n(\F),
\]
where $E$ is the matrix with all entries equal to $1$ and $I$ is the identity matrix.

\begin{proposition}[\text{\cite[Theorem 2.2.3]{bois2009generators}, \cite[Lemma 1]{chistopolskaya2018nilpotent}}]
  \label{prop:adding_with_diagonal}
Let $D$ be a consistent matrix. Then $D$ and $\1$ generate the Lie algebra $\slfrak_n(\F)$.
\end{proposition}
\begin{proof}
Let $D = \diag(\lambda_1, \lambda_2, \dots, \lambda_n)$ and observe that for every $k \geq 1$, we have
\[
  [\1, D, \stackrel{k}{\dots}, D] = 
  \left( (\lambda_j - \lambda_i)^k \right)_{i,j}.
\]
Thus the coefficients of the matrices $\1$, $[\1, D]$, \dots, $[\1, D, \stackrel{n^2-n}{\dots}, D]$ written in the standard basis $\{ E_{ij} \mid i \neq j\}$ form a Vandermonde matrix with determinant
\[
  \prod_{
    \substack{i, j, k, l \\ 
    i \neq j, k \neq l, (i,j) \neq (k,l)}
    } 
  \left( (\lambda_i - \lambda_j) - (\lambda_k -\lambda_l) \right),
\]
which is nonzero as $D$ is consistent. Therefore the Lie algebra generated by $D$ and $\1$ contains the set $\{ E_{ij} \mid i \neq j\}$, and as $[E_{ij}, E_{ji}] = E_{ii}- E_{jj}$, it thus contains all traceless matrices.
\end{proof}

\section{Bounded rank}
\label{sec:bounded_rank}

Using consistent matrices, we can easily find $2$ generators for Lie algebras $\slfrak_n(\F_q)$ of \emph{bounded rank}, meaning that $n \geq 2$ is fixed and either the characteristic $p$ of the field $\F_q$ is large enough in terms of $n$ or $q$ is large enough in terms of $p$ and $n$.

Our argument here goes through the following remark. Suppose $S = \{s_1, s_2, \dots, s_n \}$ is a subset of the vector space $\F_q$ over $\F_p$ consisting of $n$ elements. We call $S$ \emph{sharply traceless} if 
\[
  \textstyle \sum_{i = 1}^n s_i = 0
  \quad
  \text{and}
  \quad
  \dim_{\F_p} \langle S \rangle = n-1.
\]
In other words, the only linear relation up to a scalar between the vectors of $S$ is that their sum is $0$. 

\begin{lemma}
\label{lem:sharply_traceless_set_is_consistent}
A sharply traceless set of size $n$ is consistent if and only if $p \neq 2$ and $(n,p) \neq (3, 3)$.
\end{lemma}
\begin{proof}
No consistent sets exist in even characteristic. Assume now that $p \neq 2$. 
The relation violating consistency $s_i - s_j = s_k - s_l$ involves at most $4$ elements, so the set is consistent if $n \geq 5$. When $n = 4$, there are no linear relations involving at most $3$ terms, so a violating relation can be a scalar multiple of the sum zero relation only when $p = 2$. Similarly, when $n = 3$, consistency can be violated only when $p = 2$ or $p = 3$ (for example in the form $s_1 - s_2 = s_2  - s_3$, a relation that is equivalent to sum zero). Conversely, in case $(n,p) = (3,3)$, a sharply traceless set $\{ s_1, s_2, s_3 \}$ always satisfies $s_3 - s_2 = s_2 - s_1$, so it is not consistent. Finally, when $n = 2$, a sharply traceless set $\{ s, -s \}$ with $s \neq 0$ is consistent provided $p \neq 2$.
\end{proof}

It is not difficult to construct sharply traceless sets as long as the field is large enough.

\begin{proposition}
Let $p \neq 2$, $(n,p) \neq (3,3)$, $q \geq p^{n-1}$. Then $\slfrak_n(\F_q)$ is $2$-generated.
\end{proposition}
\begin{proof}
Let $k \geq n-1$ and let $\F_{p^k}$ be an extension of $\F_p$ of degree $k$ with some vector space basis $B = \{ \lambda_1, \lambda_2, \dots, \lambda_k \}$. Let
\[
  \textstyle S = \{ \lambda_1, \lambda_2, \dots, \lambda_{n-1}, - \sum_{i=1}^{n-1} \lambda_i \} \subseteq \F_{p^k}.
\]
Note that $S$ is sharply traceless, therefore it is consistent and so $\slfrak_n(\F_q)$ can be generated by $2$ elements.
\end{proof}

On the other hand, whenever we have a consistent set $\{ \lambda_1, \lambda_2, \dots, \lambda_n \}$ with entries in the integers, we can project it modulo $p$. The projected set is consistent as long as $p > 4 \max \{ \abs{\lambda_1}, \abs{\lambda_2}, \dots, \abs{\lambda_n} \}$, in which case the associated diagonal matrix $\diag(\lambda_1, \lambda_2, \dots , \lambda_q)$ together with $\1$ modulo $p$ generate $\mathfrak{sl}_n(\F_p)$ over $\F_p$, and therefore also generate $\mathfrak{sl}_n(\F_q)$ over $\F_q$.

Consistent sets with entries in positive integers are intimately related to the concept of \emph{distinct sum sets} \cite{GS}. These are sets of integers \( \{ a_i \}_{i = 0}^n \) with  
\[
  0 = a_0 < a_1 < \dots < a_n
\]  
and the property that the sums $a_i + a_j$ for $i < j$ represent each integer at most once. From such a set, one obtains the consistent set of integers
\[
  \textstyle S = \{ a_1, a_2, \dots, a_n, -\sum_{i = 1}^n a_i \},
\]  
all of whose elements are of absolute value at most $n \cdot a_n$. 

Much is known about existence of distinct sum sets. It is shown in \cite[Theorem 1 (3)]{GS} that given $n$, one can construct such a set whose every element is of size at most $n^2 + O(n^{36/23})$, and therefore the associated consistent set consists of elements with absolute values bounded by $O(n^3)$. As a consequence, for $p = \Omega(n^3)$, the Lie algebra $\mathfrak{sl}_n(\mathbb{F}_q)$ is $2$-generated. 

\section{Unbounded rank}
\label{sec:unbounded_rank}

We now construct generator matrices $\slfrak_n(\F_q)$ in \emph{unbounded rank}, meaning that $n$ is arbitrarily large and both $p$ and $q$ might be small. Note that it follows from lower bounds on distinct sum sets in residue rings \cite[Theorem 1 (4)]{GS} that no consistent sets exist over $\F_p$ when $p < n^2 - O(n)$, so the method we used in bounded rank Lie algebras will have to be adapted.

\subsection{Polynomials with consistent roots}

Our strategy in unbounded rank is to find an appropriate irreducible polynomial $f \in \F_q[x]$ whose roots form a consistent set, and thus its companion matrix $C$ belongs to $\slfrak_n(\F_q)$ and diagonalizes in some finite extension $\F$ of $\F_q$ to a consistent matrix $D = P C P^{-1} \in \slfrak_n(\F)$. The matrices $D$ and $\1$ then generate the Lie algebra $\slfrak_n(\F)$. 

\begin{example}
Let $n = 7$ and $q = 3$. Choose a representation of the finite field $\F_{3^7}$ as an extension of $\F_3$, say as
\[
  \F_{3^7} = \F_3[\omega]/(\omega^7-\omega^2+1).
\]
Now consider a random irreducible polynomial
\[
  x^7-x^5+x^3-x^2-x-1 \in \F_3[x].  
\]
Its companion matrix $C \in \slfrak_7(\F_3)$ diagonalizes over $\F_{3^7}$ as $D = P C P^{-1}$, where
% \[
%   P = 
%   \begin{pmatrix}
%     1 & \omega^{1766} & \omega^{1346} & \omega^{926} & \omega^{506} & \omega^{86} & \omega^{1852} \\
%     1 & \omega^{1776} & \omega^{1366} & \omega^{956} & \omega^{546} & \omega^{136} & \omega^{1912} \\
%     1 & \omega^{2046} & \omega^{1906} & \omega^{1766} & \omega^{1626} & \omega^{1486} & \omega^{1346} \\
%     1 & \omega^{592} & \omega^{1184} & \omega^{1776} & \omega^{182} & \omega^{774} & \omega^{1366} \\
%     1 & \omega^{682} & \omega^{1364} & \omega^{2046} & \omega^{542} & \omega^{1224} & \omega^{1906} \\
%     1 & \omega^{926} & \omega^{1852} & \omega^{592} & \omega^{1518} & \omega^{258} & \omega^{1184} \\
%     1 & \omega^{956} & \omega^{1912} & \omega^{682} & \omega^{1638} & \omega^{408} & \omega^{1364}
%     \end{pmatrix}
% \]
% and
\[
  D = \diag(\omega^{1766}, \omega^{1776}, \omega^{2046}, \omega^{592}, \omega^{682}, \omega^{926}, \omega^{956}).
\]
The matrix $D$ is consistent over the field $\F_{3^7}$, so it generates $\slfrak_7(\F_{3^7})$ together with the matrix $\1$. Observe that
\[
  P^{-1} \1 P =
  \begin{pmatrix}
    \cdot & \cdot & 2 & \cdot & 1 & 2 & 2 \\
    \cdot & 2 & \cdot & \cdot & \cdot & \cdot & \cdot \\
    \cdot & \cdot & 2 & \cdot & \cdot & \cdot & \cdot \\
    \cdot & \cdot & \cdot & 2 & \cdot & \cdot & \cdot \\
    \cdot & \cdot & \cdot & \cdot & 2 & \cdot & \cdot \\
    \cdot & \cdot & \cdot & \cdot & \cdot & 2 & \cdot \\
    \cdot & \cdot & \cdot & \cdot & \cdot & \cdot & 2
    \end{pmatrix},
\]
a matrix that surprisingly belongs to $\slfrak_7(\F_3)$.
\end{example}

The following lemma shows that it is always possible to pull back the generating pair $D$ and $\1$ to the original $\slfrak_n(\F_q)$, and moreover explains the peculiar form of the matrix $P^{-1} \1 P$ as in the example. Recall the \emph{Frobenius endomorphism}
\[
  \Fr \colon \F_{q^n} \to \F_{q^n}, 
  \quad
  \lambda \mapsto \lambda^q.
\]
The entries of $P$ and $P^{-1} \1 P$ are related to the \emph{field trace} of the extension $\F_{q^{n}}$ of $\F_q$, which computes, for a given $\alpha \in \F_{q^{n}}$, the sum of its Galois conjugates:
\[
  \textstyle \tr (\alpha) = \sum_{i=0}^{n-1} \Fr^i(\alpha) = \sum_{i=0}^{n-1} \alpha^{q^{i}}  \in \F_q.
\]

\begin{lemma}
\label{lem:coefficients}
Let $f \in \F_q[x]$ be an irreducible polynomial of degree $n$ and let $\alpha$ be a root of $f$ in a splitting field. Let $C$ be the companion matrix of $f$ with diagonal form
\[
  D = P C P^{-1} = \diag\left(\alpha, \Fr(\alpha), \dots, \Fr^{n-1}(\alpha) \right).
\]
Then
\begin{equation*}
  \label{eq: expression conjugated}
  P^{-1} \1 P = \begin{pmatrix}
  n & \tr(\alpha)  & \cdots & \tr(\alpha^{n-1}) \\
  \cdot & \cdot & \cdots & \cdot   \\
  \vdots & \vdots  & & \vdots  \\
  \cdot & \cdot  & \cdots & \cdot
   \end{pmatrix} - I.
 \end{equation*}
In particular, $P^{-1} \1 P$ belongs to $\slfrak_n(\F_q).$
\end{lemma}

\begin{proof}
The polynomial $f$ splits completely in $\F_{q^n}$. Its roots are the Galois conjugates of $\alpha$, namely $\{ \alpha, \Fr(\alpha), \dots, \Fr^{n-1}(\alpha) \}$. These are in turn the eigenvalues of $C$. Let $e_1, e_2, \dots, e_n$ be the standard basis of the vector space $(\F_{q^n})^n$. The \emph{left} eigenvector of $C$ corresponding to $\alpha$ is
\[
  v_0 = \left(1, \alpha, \alpha^2, \dots, \alpha^{n-1} \right),
\]
and the eigenvector corresponding to $\Fr^i(\alpha)$ is $v_i = \Fr^i(v_0)$, where the Frobenius endomorphism is applied componentwise. The rows of $P$ consist precisely of the vectors $v_0, v_1, \dots, v_{n-1}$. Setting $e = e_1 + e_2 + \dots + e_n$, we thus have $P e_1 = e$, hence $P^{-1} \cdot e = e_1$. This gives
\[
  e_i^{\transpose} P^{-1} e =
  \begin{cases*}
    1 & $i = 1$ \\
    0 & $i \neq 1$
  \end{cases*}
  \quad \text{and therefore} \quad
  e_i^{\transpose} P^{-1} E =
  \begin{cases*}
    e^{\transpose} & $i = 1$ \\
    0 & $i \neq 1$,
  \end{cases*}
\]
as $E$ is the matrix with all columns equal to $e$. Now, since
\[
  e^{\transpose} P =
  v_0 + v_1 + \dots + v_{n-1} =
  \left( \tr(1), \tr(\alpha), \tr(\alpha^2), \dots, \tr(\alpha^{n-1}) \right),
\]
we obtain
\[
P^{-1} E P =
\begin{pmatrix}
  n & \tr(\alpha)  & \cdots & \tr(\alpha^{n-1}) \\
  \cdot & \cdot  & \cdots & \cdot   \\
   \vdots & \vdots  & & \vdots  \\
    \cdot & \cdot  & \cdots & \cdot
   \end{pmatrix}.
\]
The lemma follows as $\1 = E - I$.
\end{proof}

\begin{corollary}
  \label{cor:consistent_splitting_field_2gen}
Let $f \in \F_q[x]$ be an irreducible polynomial of degree $n$ whose roots in a splitting field form a consistent set. Then $\slfrak_n(\F_q)$ is $2$-generated.
\end{corollary}
\begin{proof}
Let $\F$ be a splitting field for $f$. Let $C$ be the companion matrix of $f$. Since the roots of $f$ form a consistent set, we have $C \in \slfrak_n(\F_q)$ and the eigenvalues of $C$, being the roots of $f$, are all distinct in $\F$. Therefore $C$ diagonalizes to a consistent matrix $D = P C P^{-1}$ for some $P \in \GL_n(\F)$, which generates $\slfrak_n(\F)$ together with the matrix $\1$. Thus, according to \cite[Proposition 1.1.3]{bois2009generators}, there exist $n^2 - 1$ Lie monomials $\{ g_i(X,Y) \}_i$ in two variables whose evaluations in $(D, \1)$ span the vector space $\slfrak_n(\F)$ over $\F$. Since Lie polynomials are equivariant (meaning that if $g(X,Y)$ is a Lie polynomial, then $P^{-1} g(X,Y) P = g(P^{-1} X P, P^{-1} Y P)$) the matrices $g_i(C, P^{-1} \1 P ) = P^{-1} g_i(D, \1) P$ also span $\slfrak_n(\F)$ over $\F$. Now, both $C$ and $P^{-1} \1 P$ belong to $\slfrak_n(\F_q)$, and as each $g_i$ is a Lie monomial, its evaluation $g_i(C, P^{-1} \1 P)$ is in $\slfrak_n(\F_q)$. Therefore $C$ and $P^{-1}\1P$ generate $\slfrak_n(\F_q)$ as a Lie algebra over $\F_q$.
\end{proof}

\subsection{Normal elements}

We proceed by exhibiting an irreducible polynomial $f \in \F_q[x]$ of degree $n$ whose roots form a consistent set in the case when $\slfrak_n(\F_q)$ is \emph{simple}, i.e., the characteristic $p$ and degree $n$ are coprime. 

In order to do so, we make use of the concept of a \emph{normal element}, which is an element $\alpha \in \F_{q^n}$ whose Galois conjugates $\{ \alpha, \alpha^q, \dots, \alpha^{q^{n-1}} \}$ form a basis of $\F_{q^n}$ as a vector space over $\F_q$. It is well-known that normal elements exist in any finite field \cite{Hensel}.

Given a normal element $\alpha$, consider $\beta = \alpha - \tr(\alpha)/n$. One can indeed construct this element since $n$ is invertible in $\F_q$. Note that $\beta$ is a traceless element and as $\alpha$ is normal, the set of Galois conjugates of $\beta$ forms a consistent set.\footnote{The Galois conjugates of $\beta$ are $\alpha^{p^i} - \tr(\alpha)/n$, so their pairwise differences are the same as the pairwise differences of Galois conjugates of $\alpha$.} Taking $f$ to be the minimal polynomial of $\beta$ then does the job.

\begin{corollary}
\label{cor:coprime}
Let $p \neq 2$ and $n$ coprime to $p$. Then $\slfrak_n(\F_q)$ is $2$-generated.
\end{corollary}

\subsection{Sharply traceless elements}

In case $p$ divides $n$, normal elements still exist, but there is no evident way of producing a \emph{traceless} matrix from them. 

In order to proceed, we will replace normal elements with their traceless version. To be more precise, we call an element $\alpha \in \F_{q^n}$ \emph{sharply traceless} if the set of its Galois conjugates over $\F_q$ forms a sharply traceless set. In other words, $\alpha$ is traceless, meaning that the sum of the Galois conjugates of $\alpha$ is zero, and this is the only linear relation up to a scalar between the Galois conjugates.

\begin{theorem}
\label{thm:sharply_traceless}
Sharply traceless elements exist in any extension $\F_{q^n}/\F_q$.
\end{theorem}

Generation of $\slfrak_n(\F_q)$ now follows as in the previous section.

\begin{corollary}
\label{cor:non-coprime}
Let $p \neq 2$ and $(n,p) \neq (3,3)$. Then $\slfrak_n(\F_q)$ is $2$-generated.
\end{corollary}
\begin{proof}
A sharply traceless element gives a sharply traceless set of its Galois conjugates. The result follows from Lemma \ref{lem:sharply_traceless_set_is_consistent}.
\end{proof}

Our proof of \Cref{thm:sharply_traceless} relies on viewing the finite field $\F_{q^n}$ as a module over the modular group algebra of the Galois group $G = \Gal(\F_{q^n}/\F_q)$, generated by the Frobenius endomorphism $\Fr$. More precisely, we consider the ring $\F_q[G] = \F_q[x]/(m(x))$, where $m(x) = x^n - 1$ is the minimal and at the same time characteristic polynomial of $\Fr$ (see \cite[proof of Theorem 2.35]{LN}). We can express the trace in terms of the polynomial 
\[
  \Phi(x) = \frac{m(x)}{x-1} = 
  x^{n-1} + x^{n-2} + \dots + 1,
\]
since, for an element $\alpha \in \F_{q^n}$, we have $\tr(\alpha) = \Phi(\Fr)\alpha$. Consider now the module homomorphism
\[
  \eta \colon \F_q[G] \to \F_{q^n}, \quad
  t \mapsto t \cdot \alpha.
\]
Note that $\eta$ is an isomorphism precisely when $\alpha$ is a normal element in $\F_{q^n}$. Write $t = \sum_{i = 0}^{n-1} t_i \Fr^i$ and let $\epsilon(t) = \sum_{i = 0}^{n-1} t_i$ be the augmentation of $t$. Then
\[
  \textstyle \tr(t \cdot \alpha) = 
  \sum_{i = 0}^{n-1} t_i \tr(\Fr^i(\alpha)) = 
  \epsilon(t) \cdot \tr(\alpha).
\]
For a normal $\alpha$, we have $\tr(\alpha) \neq 0$, and so elements of trace $0$ in $\F_{q^n}$ correspond precisely to the image of the augmentation ideal under $\eta$. This ideal is generated by $\Fr - 1$, so traceless elements in $\F_{q^n}$ are generated by $(\Fr - 1)(\alpha) = \alpha^q - \alpha$ as a module over $\F_q[G]$. The sharply traceless ones among these can be understood as follows.

\begin{lemma}
  Let $\alpha$ be a normal element in $\F_{q^n}$ and $\pi$ a polynomial in $\F_q[x]$. Then $\beta = \pi(\Fr) (\Fr - 1) (\alpha)$ is sharply traceless if and only if $\pi$ is coprime to $\Phi$.
\end{lemma}
\begin{proof}
Suppose first that $\pi$ is coprime to $\Phi$. By B\'ezout, there are polynomials $a, b$ in $\F_q[x]$ so that $a \pi + b \Phi = 1$. Evaluating in $\Fr$ and applying to $(\Fr - 1)(\alpha)$, we obtain $a(\Fr)(\beta) = (\Fr - 1)(\alpha)$. Hence if some polynomial $c$ in $\F_q[x]$ satisfies $c(\Fr)(\beta) = 0$, then also $c(\Fr)(\Fr - 1)(\alpha) = 0$, and since $\alpha$ is normal, it follows that $\Phi$ divides $c$. Thus $\beta$ is sharply traceless.

Conversely, suppose that $\beta = \pi(\Fr) (\Fr - 1)(\alpha)$ is sharply traceless element. Note
\[
  \frac{\Phi}{\gcd(\pi, \Phi)}(\Fr)(\beta) = m(\Fr) \frac{\pi}{\gcd(\pi, \Phi)}(\Fr) (\alpha) = 0.
\]
As $\beta$ is sharply traceless, it follows that $\gcd(\pi, \Phi) = 1$.
\end{proof}

Note that two polynomials $\pi, \pi'$ in $\F_q[x]$ determine the same sharply traceless element $\beta$ if and only if $(\Fr - 1)(\pi - \pi')(\Fr) \alpha = 0$, which is equivalent to saying that $\Phi$ divides $\pi - \pi'$.

\begin{corollary}
  Sharply traceless elements in $\F_{q^n}$ are in bijective correspondence with the group of units of $\F_q[x]/(\Phi(x))$.\footnote{The same argument gives that normal elements in $\F_{q^n}$ are in bijective correspondence with the group of units of $\F_q[x]/(m(x))$.}
\end{corollary}

The number of these elements can be computed as follows. Let $\Div(\Phi)$ be the set of prime factors of $\Phi$. For a given $\pi \in \Div(\Phi)$, let $\mult(\pi)$ be the multiplicity of $\pi$ in $\Phi$. Then 
\[
  \frac{\F_q[x]}{(\Phi(x))} \cong
  \prod_{\pi \in \Div(\Phi)} \frac{\F_q[x]}{(\pi(x)^{\mult(\pi)})}.
\]
The invertible elements in each factor are polynomials that are coprime to $\pi$, so their number is 
\[
  q^{\deg(\pi) \mult(\pi)} - q^{\deg(\pi) (\mult(\pi) - 1)} =
  q^{\deg(\pi) \mult(\pi)} (1 - q^{- \deg(\pi)}).
\]
Note that $n-1 = \deg(\Phi) = \sum_{\pi \in \Div(\Phi)}  \deg(\pi) \mult(\pi)$, implying the following.

\begin{corollary}
The number of sharply traceless elements in $\F_{q^n}$ is
\[
  q^{n - 1} \prod_{\pi \in \Div(\Phi)} (1 - q^{-\deg(\pi)}).
\]    
\end{corollary}

For example, when $n = 2$, we have $\Phi(x) = x + 1$, so we have $q - 1$ sharply traceless elements in $\F_{q^2}$, and this is precisely the same as the number of nonzero traceless elements. On the other hand, when $n = q = p$, we have $\Phi(x) = (x-1)^{p-1}$, so the number of sharply traceless elements in $\F_{p^p}$ is $p^{p-2} (p-1) = p^{p-1} - p^{p-2}$, whereas the number of nonzero traceless elements equals $p^{p-1} - 1$. In general, for a fixed $n$, the number given in the corollary is a polynomial in $q$ of degree $n-1$, implying the following.

\begin{corollary}
For any fixed $n$, a uniformly random traceless element of $\F_{q^n}$ is sharply traceless with probability tending to $1$ as $q$ tends to infinity.
\end{corollary}

An alternative (but longer) route\footnote{We thank the referee for pointing out the simpler argument using group algebras.} to proving \Cref{thm:sharply_traceless} is by mimicking the usual approach for finding a normal element, compare with \cite{GG} or \cite{LN} and the closely related notion of a cyclic element from \cite{neumann1995cyclic}. 

\section{Even characteristic}
\label{sec:even_characteristic}

Our method for proving $2$-generation in the previous sections used consistent matrices. These do not exist when the field is of even characteristic. In this situation, we instead rely on \emph{semiconsistent} matrices. These are diagonal matrices
 \[ 
  D = \diag(\lambda_1, \dots, \lambda_n) \in \slfrak_n(\F_q) 
\]
that satisfy the condition
\[
\lambda_i + \lambda_j = \lambda_k + \lambda_l \ \textup{ if and only if } \ \{ i, j \} = \{ k, l \} \text{ or } ( i, k ) = ( j, l).\footnote{Compare with \eqref{eq:consistency_def}.}
\]
For instance, if $\alpha \in \F_{q^n}$ is a sharply traceless element, the diagonal matrix $\diag(\alpha, \Fr(\alpha), \dots, \Fr^{n-1}(\alpha))$ is semiconsistent in $\slfrak_n(\F_{q^n})$ as long as $n \neq 4$.

\begin{proposition}
Let $p = 2$ and $n \neq 2, 4$. Then $\slfrak_n(\F_q)$ is $2$-generated.
\end{proposition}
\begin{proof}
It suffices to prove that $\slfrak_n(\F_2)$ is $2$-generated. 
Let $\alpha \in \F_{2^n}$ be a sharply traceless element with minimal polynomial $f \in \F_2[x]$ of degree $n$. Let $C \in \slfrak_n(\F_2)$ be the companion matrix of $f$ with diagonal form 
\[ 
  D = P C P^{-1} = \diag\left( \alpha, \Fr(\alpha), \dots, \Fr^{n-1}(\alpha) \right) \in \slfrak_n(\F_{2^n}).
\]
Since $\alpha$ is sharply traceless, $D$ is semiconsistent. 
 
Now let $X = (x_{ij})_{i,j}$ be an arbitrary matrix in $\slfrak_n(\F_{2^n})$. We have
\[ 
  [X, D, \stackrel{k}{\dots}, D] =  \sum_{i < j} (\lambda_i + \lambda_j)^{k} \left( x_{ij} E_{ij} + x_{ji} E_{ji} \right),
\]
so arguing as in the proof of \Cref{prop:adding_with_diagonal}, the matrix $A_{ij} = x_{ij} E_{ij} + x_{ji} E_{ji}$ for $i \neq j$ is contained in the Lie algebra generated by $X$ and $D$. 

Let us consider $X = P (E_{21} + E_{13}) P^T$. Recall from \Cref{lem:coefficients} that the rows of $P$ are $v_0, v_1, \dots, v_{n-1}$, where $v_0 = (1, \alpha, \alpha^2, \dots, \alpha^{n-1})$ and $v_i = \Fr^i(v_0)$. Thus
\[
  e_k^{\transpose} P E_{ij} P^{T} e_l = 
  (e_k^{\transpose} P e_i) (e_j^{\transpose} P^{T} e_l) = 
  \Fr^{k-1}(\alpha^{i-1}) \Fr^{l-1}(\alpha^{j-1})
\]
and so
\[
  x_{kl} = \Fr^{k-1}(\alpha) \Fr^{l - 1}(1) + \Fr^{k-1}(1) \Fr^{l - 1}(\alpha^2)
  = \Fr^{k-1}(\alpha) + \Fr^{l}(\alpha).
\]
Since we are in even characteristic and $\alpha$ is sharply traceless, $x_{kl} = 0$ if and only if $l \equiv k - 1 \pmod{n}$. In particular,  $E_{n1} = x_{n1}^{-1}A_{n1}$ and $E_{i, i+1} =  x_{i, i+1}^{-1}A_{i, i+1}$ for $1 \leq i \leq n-1$ are in the Lie algebra generated by $D$ and $X.$ Since the matrices $E_{n1}$ and $E_{i,i+1}$ for $1 \leq i \leq n-1$ generate $\slfrak_n(\F_{2^n})$, the matrices $D$ and $X$ generate it as well, and so the matrices $C= P^{-1} D P$ and $P^{-1} X P = (E_{21} + E_{13})P^{\transpose} P$ also generate it. 

Note, however, that both $C$ and $(E_{21} + E_{13})P^{\transpose} P$ belong to $\slfrak_n(\F_2)$, since
\[
  e_{k}^{\transpose} P^{\transpose} P e_l =
  (P e_k)^{\transpose} (P e_l) =
  \sum_{i = 0}^{n-1} \Fr^i(\alpha^{k-1}) \Fr^i(\alpha^{l-1}) =
  % \sum_{i = 0}^{n-1} \Fr^i(\alpha^{k+l-2}) =
  \tr(\alpha^{k+l-2})
\]
for any $k,l$. Therefore, as in the proof of \Cref{cor:consistent_splitting_field_2gen}, the matrices $C$ and $(E_{21} + E_{13})P^{\transpose} P$ generate the Lie algebra $\slfrak_n(\F_2)$.
 \end{proof}

We still have to take care of the Lie algebras of rank $n=2$ and $n = 4$. Clearly $\slfrak_2(\F_2)$ is $2$-generated, for instance by $E_{12}$ and $E_{21}$. We deal with $\slfrak_4(\F)$ in \Cref{sec:exceptions}.

\section{Exceptional cases}
\label{sec:exceptions}

There are two exceptions that our methods with sharply traceless elements do not cover, namely the families $\slfrak_3(\F_q)$ with $\F_q$ of characteristic $3$ and $\mathfrak{sl}_4(\F_q)$ with $\F_q$ of characteristic $2.$ We will prove that these algebras are in fact \emph{not} $2$-generated, even over \emph{infinite} fields of corresponding characteristic. This failure of $2$-generation comes as a surprise, since this is a property enjoyed throughout simple algebraic structures over finite and infinite fields as pointed out in \Cref{sec:introduction}. We will show that this is a phenomenon that heavily relies both on the modularity of the situation and on smallness of the rank.

\subsection{\texorpdfstring{$(n,p) = (3,3)$}{(n,p) = (3,3)}}

\begin{theorem}
  \label{thm:psl3_characteristic_3}
  Let $\F$ be a field of characteristic $3$. Then any two elements of $\slfrak_3(\F)$ generate a Lie algebra of dimension at most $4$. In particular,
  $\slfrak_3(\F)$ is not $2$-generated. It is, however, $3$-generated.
\end{theorem}

Note that our result diverges from \cite[Theorem 1]{chistopolskaya2018nilpotent}, whose argument appears to hinge on the existence of a consistent matrix, but these do not exist in $\slfrak_3(\F)$ when $\F$ is of characteristic $3$, see \Cref{lem:sharply_traceless_set_is_consistent}.

%\footnote{Additionally so due to \cite[Theorem 1]{chistopolskaya2018nilpotent}, whose proof assumes the existence of a consistent matrix, but these do not exist in $\slfrak_3(\F)$ when $\F$ is of characteristic $3$, see \Cref{lem:sharply_traceless_set_is_consistent}.}

The claim that $\slfrak_3(\F)$ is $3$-generated follows immediately from the fact that over the prime field $\F_3$ of $\F$, the matrices
\begin{equation}
\label{eq:3_generators}
A = \begin{pmatrix} 1 & \cdot & \cdot \\ \cdot & -1 & \cdot \\ \cdot & \cdot & \cdot \end{pmatrix}, \
B = \begin{pmatrix} \cdot & \cdot & 1 \\\cdot & \cdot & 1 \\ \cdot & \cdot & \cdot \end{pmatrix}, \
B^{\transpose} = \begin{pmatrix} \cdot & \cdot & \cdot \\ \cdot & \cdot & \cdot \\ 1 & 1 & \cdot \end{pmatrix}
\end{equation}
generate $\slfrak_3(\F_3)$, and so they also generate $\slfrak_3(\F)$ over $\F$. Indeed, $[A,B] - B = E_{23}$ and $[A,B^{\transpose}] - B^{\transpose} = E_{31}$. Thus the Lie algebra generated by $A,B,B^{\transpose}$ contains $E_{13} = B - E_{23}$, $E_{32} = B^{\transpose} - E_{31}$, and therefore $[E_{23}, E_{31}] = E_{21}$ and $[E_{13}, E_{32}] = E_{12}$.

% \subsection{Trace identity}

We shall derive \Cref{thm:psl3_characteristic_3} from the following curious identity in the quotient $\pslfrak_3(\F)$ of the Lie algebra $\slfrak_3(\F)$ by the central ideal consisting of scalar multiples of the identity.
\begin{proposition}
  \label{prop:trace_formula}
For any $x, y \in \slfrak_3(\F)$, the following identity holds in $\pslfrak_3(\F)$:
  \[
  [x,y,y] = - \tr(y^2) x + \tr(xy) y.
  \]
\end{proposition}

This identity implies that for any $x, y \in \pslfrak_3(\F)$, the vector space spanned by $x, y, [x, y]$ is closed with respect to Lie bracketing with $x$ and $y$. In other words, the Lie algebra generated by $x$ and $y$ is of dimension at most $3$. Thus $\pslfrak_3(\F)$ is not $2$-generated. Furthermore, it follows from the same identity that any two elements in $\slfrak_3(\F)$ generate a Lie algebra of dimension at most $4$. Note that this bound is sharp, for example the Lie algebra generated by the matrices $A$ and $B$ from \eqref{eq:3_generators} is of dimension $4$.

\begin{proof}[Proof of \Cref{prop:trace_formula}]
Let $x \in \slfrak_3(\F)$. The characteristic polynomial of $x$ is $t^3 + \alpha(x) t - \det(x)$, where $\alpha(x)$ is a form in $x$, computable as the elementary symmetric polynomial $e_2$ in the eigenvalues of $x$. By Newton's identities, $e_2$ can be expressed as $(p_1^2 - p_2)/2$, where $p_i$ are power sums. As $\tr(x) = 0$ and $\F$ is of characteristic $3$, it follows that $\alpha(x) = p_2(x) = \tr(x^2)$. Therefore we have
\begin{equation}
  \label{eq:char_poly_33}
  x^3 + \tr(x^2) x = 0
\end{equation}
in $\pslfrak_3(\F)$. This identity holds over any field of characteristic $3$, as well as any subquotient of such a field. We can thus multilinearize this identity by using it in the Lie algebra $\pslfrak_3(\F) \otimes_{\F} \F[T]/(T^2)$ with the element  $Tx + y$ for any $x,y \in \slfrak_3(\F)$.\footnote{We thank Alexander Premet who showed us this trick. It substantially shortens our previous argument.} Note that, by Jacobson's formula \cite{jacobson1979lie},
\[
  (Tx + y)^3 = 
  T^3 x^3 + y^3 + [Tx,y,y] - [Tx,y,Tx] = 
  y^3 + T [x,y,y]
\]
since $T^2 = 0$, and similarly
\[
  \tr((Tx + y)^2) = - T \tr(xy) + \tr(y^2).
\]
By inserting $Tx + y$ into \eqref{eq:char_poly_33}, we thus obtain $[x,y,y] - \tr(xy) y + \tr(y^2) x = 0$.
\end{proof}

In fact, when the \emph{$u$-invariant}\footnote{The largest integer $u(\F) \geq 0$ such that there exists an anisotropic quadratic form of dimension $u(\F)$ over $\F$, or is infinity if no such largest integer exists.} of $\F$  is at most $2$, the identity in \Cref{prop:trace_formula} further implies that for generic choices of $x$ and $y$, the Lie algebra they generate is isomorphic to $\slfrak_2(\F)$. This occurs, for example, when $\F$ is finite or algebraically closed (see \cite[Example 36.2]{EKM}).

\begin{corollary}
\label{cor:generic_sl2_in_sl3}
Let $\F$ be a field of characteristic $3$ such that $u(\F) \leq 2$. Let $x, y \in \slfrak_3(\F)$. Let $G \in \Mat_3(\F)$ be the Gram matrix of $x, y, [x,y]$ under the trace form. Then the Lie algebra generated by the images of $x$ and $y$ in $\pslfrak_3(\F)$ is isomorphic to $\slfrak_2(\F)$ if and only if $\det G \neq 0$.\footnote{It is not difficult to see that $\det G = (\tr(x^2)\tr(y^2)-(\tr(xy))^2)(\tr(x^2y^2)-\tr((xy)^2))$.}
\end{corollary}
\begin{proof}
Suppose first that $G$ is nondegenerate. Hence $x,y,[x,y]$ are linearly independent. The restriction of the trace form to the span of $x,y,[x,y]$ is thus a nondegenerate quadratic form of dimension $3$. This form is isotropic, so by the Witt decomposition theorem,  it contains a hyperbolic plane spanned by some $e,f$. Thus $\tr(e^2) = \tr(f^2) = 0$ and $\tr(ef) = 1$. Let $h = [e,f]$. Using \Cref{prop:trace_formula}, we then compute $[h,e] = 2e$ and $[h,f] = -2f$ in $\pslfrak_3(\F)$. These formulas further imply that $e,f,h$ are linearly independent. After passing to $\pslfrak_3(\F)$, the Lie algebra generated by $x$ and $y$ is thus isomorphic to $\slfrak_2(\F)$.

Assume now that $G$ is degenerate.
We can assume that the Lie algebra generated by the images of $x,y$ in $\pslfrak_3(\F)$ is of dimension $3$. Hence $x,y,[x,y]$ are linearly independent and the restriction of the trace form to their span $L = \langle x, y, [x,y] \rangle$ is degenerate. Thus there is a nonzero $z \in L$ with $\tr(zx) = \tr(zy) = \tr(z[x,y]) = 0$. Using \Cref{prop:trace_formula}, we conclude that $\ad_z$ is nilpotent of degree $2$ on $\pslfrak_3(\F)$. The Lie algebra $\slfrak_2(\F)$, however, does not contain nonzero nilpotent elements of degree $2$ when $\F$ is of odd characteristic. This completes the proof.
\end{proof}

 In particular, when $\F$ is finite the probability that a uniformly random tuple $x,y \in \slfrak_3(\F)$ belongs to the variety $\det G = 0$ is at most $8/\abs{\F}$ by the Schwartz-Zippel lemma, which tends to $0$ as $\abs{\F}$ tends to infinity.

\subsection{\texorpdfstring{$(n,p) = (4,2)$}{(n,p) = (4,2)}}

This case is handled using the same method.

\begin{theorem}
  \label{thm:psl4_characteristic_2}
  Let $\F$ be a field of characteristic $2$. Then any two elements of $\slfrak_4(\F)$ generate a Lie algebra of dimension at most $9$. In particular,
  $\slfrak_4(\F)$ is not $2$-generated. It is, however, $3$-generated.
\end{theorem}

It is not difficult to find $3$ matrices in $\slfrak_4(\F)$ that generate the whole Lie algebra. One can take, for example,
\begin{equation}
  \label{eq:42_ABBT}
A = \begin{pmatrix} 
  1 & \cdot & \cdot & \cdot \\
  \cdot & 1 & \cdot & \cdot \\
  \cdot & \cdot & \cdot & \cdot \\
  \cdot & \cdot & \cdot & \cdot \end{pmatrix},\ 
B = \begin{pmatrix}
  \cdot & \cdot & 1 & 1 \\
  1 & \cdot & \cdot & \cdot \\
  \cdot & \cdot & \cdot & \cdot \\
  1 & 1 & \cdot & \cdot \end{pmatrix},\
B^{\transpose} = \begin{pmatrix}
  \cdot & 1 & \cdot & 1 \\
  \cdot & \cdot & \cdot & 1 \\
  1 & \cdot & \cdot & \cdot \\
  1 & \cdot & \cdot & \cdot \end{pmatrix}.
\end{equation}
Indeed, the Lie algebra $L$ generated by these elements contains $[A,B] + B = E_{21}$ and thus also $[E_{21}, B^{\transpose}] + A = E_{24} \in L$. Symmetrically, we obtain $E_{12}, E_{42} \in L$. Therefore $[E_{42}, E_{21}] = E_{41} \in L$ and symmetrically $E_{14} \in L$. It follows that $[E_{41}, E_{14} + B] = E_{43} \in L$ and thus $[E_{24}, E_{43}] = E_{23} \in L$. Symmetrically we obtain $E_{32} \in L$ and hence $[E_{32}, E_{21}] = E_{31} \in L$. This means that $L$ contains all nondiagonal elementary matrices, therefore $L = \slfrak_4(\F)$.

The difficult part of the proof of \Cref{thm:psl4_characteristic_2} again follows from an incidental identity that relies on just the right smallness of the rank and even characteristic.

\begin{proposition}
  \label{prop:42_formula}
There are forms $a, b, c, d$ on $\slfrak_4(\F)$ so that for any $x, y \in \slfrak_4(\F)$, the following identity holds in $\pslfrak_4(\F)$:
  \[
  [x,y,y,y] = a(x,y) x + b(x,y) y + c(x,y) [x,y] + d(x,y) y^2.
  \]
\end{proposition}

For any $x, y \in \pslfrak_4(\F)$, the Lie algebra generated by $x$ and $y$ contains the following elements:
\begin{equation}
  \label{eq:42_8elements}
  x, y, [x,y], [x,y,y], [y,x,x], [x,y,y,y], [y,x,x,x], [x,y,x,y].
\end{equation}
Since the Lie algebra $\pslfrak_4(\F)$ is $2$-restricted, we have $\ad_{z^2} = (\ad_z)^2$ for any $z$. This implies that $[x,y,y,y] = [x, y, y^2]$ and $[x,y,x,y] = [y,x,x,y] = [y, x^2, y] = [x^2, y, y] = [x^2, y^2]$. It then follows from \Cref{prop:42_formula} that the elements \eqref{eq:42_8elements} all belong to the vector space spanned by
\[
  x, y, [x,y], [x,y^2], [y, x^2], x^2, y^2, [x^2, y^2],
\]
and at the same time this vector space is closed with respect to Lie bracketing with $x$ and $y$. In other words, the Lie algebra generated by $x$ and $y$ is of dimension at most $8$. Thus $\pslfrak_4(\F)$ is not $2$-generated. Furthermore, it follows from the same identity that any two elements in $\slfrak_4(\F)$ generate a Lie algebra of dimension at most $9$. Note that this bound is sharp, for example the Lie algebra generated by the matrices $B$ and $B^{\transpose}$ from \eqref{eq:42_ABBT} is of dimension $9$.

\begin{proof}[Proof of \Cref{prop:42_formula}]
  Let $x \in \slfrak_4(\F)$. The characteristic polynomial of $x$ is $t^4 + \alpha(x) t^2 - \beta(x) t + \det(x)$, where $\alpha, \beta$ are forms in $x$ with coefficients in $\F_2$.\footnote{These forms are computable as elementary symmetric polynomials $e_2, e_3$ in the eigenvalues of $x$. As we are in characteristic $2$, these polynomials are not expressible in terms of power sums, so the formula is not as clear as the one in the previous section.} Therefore we have
  \begin{equation}
    \label{eq:42_char_poly}
    x^4 + \alpha(x) x^2 - \beta(x) x = 0
  \end{equation}
  in $\pslfrak_4(\F)$. As in the proof of \Cref{prop:trace_formula}, we multilinearize this identity by using it with the element $Tx + y$ with $T^2 = 0$. By Jacobson's formula, we have
  \[
    (Tx + y)^2 = y^2 + T [x,y] 
    \quad \text{and so} \quad
    (Tx + y)^4 = y^4 + T [x,y,y^2].
  \]
  Note that $\alpha(Tx + y), \beta(Tx + y)$ are polynomials in $T$, since they are expressible as coefficients of the polynomial $\det(tI - Tx - y)$ in $t$. Write $\alpha(Tx + y) = c + Td$ in $\F[T]/(T^2)$ and similarly $\beta(Tx + y) = a + Tb$, where $a,b,c,d$ are forms on $\slfrak_4(\F)$ evaluated at $x,y$. By inserting $Tx+y$ into \eqref{eq:42_char_poly}, we thus obtain
  \[
    [x,y,y^2] + d y^2 + c [x,y] + a x + b y = 0,
  \]
  as required.
\end{proof}

A similar phenomenon to \Cref{cor:generic_sl2_in_sl3} appears to hold in this case as well, although we have not been able to prove it. Generically, the Lie algebra generated by two elements in $\pslfrak_4(\F)$ seems to be isomorphic to $\slfrak_3(\F)$.

\bibliographystyle{alpha}
\bibliography{biblio}

\end{document}